\documentclass[10pt,twoside,leqno]{siamltex}
\usepackage{amsmath,amssymb,amsfonts}							
\usepackage{epsfig}
\usepackage{float}
\usepackage{enumitem}
\usepackage[table]{xcolor}									
\usepackage{hyperref}

\newcommand{\bb}[1]{\mathbb{#1}}								
\newcommand{\ii}{\textup{i}}									

\newcommand{\sr}[1]{\rho\left(#1\right)}							
\newcommand{\sig}[1]{\sigma \left( #1 \right)}						

\newcommand{\mat}[2]{M_{#1}(#2)}								
\newcommand{\jordan}[2]{J_{#1}{\left( #2 \right)}}					

\newcommand{\inv}[1]{#1^{-1}}								

\newcommand{\hyp}[2]{\hyperref[#1]{{\rm #2 \ref*{#1}}}}				

\newtheorem{rem}[theorem]{Remark}
\newtheorem{ex}[theorem]{Example}

\begin{document}



\title{Matrix functions that Preserve the Strong Perron-Frobenius Property\thanks{}}

\author{
Pietro Paparella\thanks{Division of Engineering and Mathematics, University of Washington Bothell, Bothell, Washington 98011, USA (pietrop@uw.edu)}
}

\pagestyle{myheadings}
\markboth{Pietro Paparella}{Matrix functions that Preserve the Strong Perron-Frobenius Property}
\maketitle

\begin{abstract}
In this note, we characterize matrix functions that preserve the strong Perron-Frobenius property using the \textit{real Jordan canonical form} of a real matrix.  
\end{abstract}

\begin{keywords}
Matrix function, Real Jordan Canonical form, Perron-Frobenius Theorem, Eventually Positive Matrix
\end{keywords}

\begin{AMS}
15A16, 15B48, 15A21.
\end{AMS}

\section{Introduction}

A real matrix has the \textit{Perron-Frobenius property} if its spectral radius is a positive eigenvalue corresponding to an entrywise nonnegative eigenvector. The \textit{strong Perron-Frobenius property} further requires that the spectral radius is simple; that it dominates in modulus every other eigenvalue; and that it has an entrywise positive eigenvector.

In \cite{mw1979}, Micchelli and Willoughby characterized matrix functions that preserve \textit{doubly nonnegative matrices}. In \cite{gkr2013}, Guillot et al.~used these results to solve the \textit{critical exponent conjecture} established in \cite{jlw2011}. In \cite{bh2008}, Bharali and Holtz characterized entire functions that preserve nonnegative matrices of a fixed order and, in addition, they characterized matrix functions that preserve nonnegative block triangular, circulant, and symmetric matrices. In \cite{es2010}, Elhashash and Szyld characterized entire functions that preserve sets of generalized nonnegative matrices. 

In this work, using the characterization of a matrix function via the \textit{real Jordan canonical form} established in \cite{mpt2014}, we characterize matrix functions that preserve the strong Perron-Frobenius property. Although our results are similar to those presented in \cite{es2010}, the assumption of entirety of a function is dropped in favor of analyticity in some domain containing the spectrum of a matrix.

\section{Notation}

Denote by $\mat{n}{\bb{C}}$ (respectively, $\mat{n}{\bb{R}}$) the algebra of complex (respectively, real) $n \times n$ matrices. Given $A \in \mat{n}{\bb{C}}$, the \textit{spectrum} of $A$ is denoted by $\sig{A}$, and the \textit{spectral radius} of $A$ is denoted by $\sr{A}$. 

 The \textit{direct sum} of the matrices $A_1, \dots, A_k$, where $A_i \in \mat{n_i}{\bb{C}}$, denoted by $A_1 \oplus \dots \oplus A_k$, $\bigoplus_{i=1}^k A_i$, or $\diag{(A_1,\dots,A_k)}$, is the $n \times n$ matrix 
\[ \begin{bmatrix} A_1 \\ & \ddots \\ & & A_k \end{bmatrix},~n = \sum_{i=1}^k n_i. \]

For $\lambda \in \bb{C}$, $\jordan{n}{\lambda}$ denotes the $n \times n$ \textit{Jordan block} with eigenvalue $\lambda$. For $A \in \mat{n}{\bb{C}}$, denote by $J = \inv{Z} A Z = \bigoplus_{i=1}^t \jordan{n_i}{\lambda_i} = \bigoplus_{i=1}^t J_{n_i}$, where $\sum n_i = n$, a Jordan canonical form of $A$. Denote by $\lambda_1,\dots,\lambda_s$ the \textit{distinct} eigenvalues of $A$, and, for $i=1,\dots,s$, let $m_i$ denote the \textit{index} of $\lambda_i$, i.e., the size of the largest Jordan block associated with $\lambda_i$. Denote by $\ii$ the imaginary unit, i.e., $\ii := \sqrt{-1}$.
 
A domain $\mathcal{D}$ is any open and connected subset of $\bb{C}$. We call a domain \textit{self-conjugate} if $\bar{\lambda} \in \mathcal{D}$ whenever $\lambda \in \mathcal{D}$ (i.e., $\mathcal{D}$ is symmetric with respect to the real-axis). Given that an open and connected set is also path-connected, it follows that if $\mathcal{D}$ is self-conjugate, then $\bb{R} \cap \mathcal{D} \neq \emptyset$.  

\section{Background}

Although there are multiple ways to define a matrix function (see, e.g., \cite{h2008}), our preference is via the Jordan Canonical Form.

\begin{definition} \label{def:functionvalues}
{\rm Let $f : \bb{C} \longrightarrow \bb{C}$ be a function and denote by $f^{(j)}$ the $j$th derivative of $f$. The function $f$ is said to be \textit{defined on the spectrum of $A$} if the values
\begin{align*}
\begin{array}{c c c}
f^{(j)}(\lambda_i), & j=0,\dots,m_i-1, & i=1,\dots,s,
\end{array}
\end{align*}
called \textit{the values of the function $f$ on the spectrum of $A$}, exist.}
\end{definition}

\begin{definition}[Matrix function via Jordan canonical form] 
{\rm If $f$ is defined on the spectrum of $A \in \mat{n}{\bb{C}}$, then
\begin{align*} 
f(A) := Z f(J) \inv{Z} = Z \left( \bigoplus_{i=1}^t f(J_{n_i}) \right) \inv{Z}, 
\end{align*}
where
\begin{align}
f(J_{n_i}) := 
\begin{bmatrix} 
f(\lambda_i) & f'(\lambda_i) & \dots   & \frac{f^{(n_i-1)}(\lambda_i)}{(n_i - 1)!} 	\\
	          & f(\lambda_i)  & \ddots & \vdots 						\\
	          &                       & \ddots & f'(\lambda_i)						\\
	          &  	  	   &             & f(\lambda_i)
\end{bmatrix}. \label{fjb}
\end{align}}
\end{definition}

The following theorem is well-known (for details see, e.g., \cite{hj1990}, \cite{lt1985}; for a complete proof, see, e.g., \cite{glr1986}). 

\begin{theorem}[Real Jordan canonical form] \label{thm:rjcf} 
If $A \in \mat{n}{\bb{R}}$ has $r$ real eigenvalues (including multiplicities) and $c$ complex conjugate pairs of eigenvalues (including multiplicities), then there exists an invertible matrix $R \in \mat{n}{\bb{R}}$ such that 
\begin{align}
\inv{R} A R =
\begin{bmatrix} 
\bigoplus_{k=1}^{r} J_{n_k} (\lambda_k) & \\ & \bigoplus_{k = r + 1}^{r + c} C_{n_k} (\lambda_k)  
\end{bmatrix},															\label{realjordform}	
\end{align}
where:
\begin{enumerate}
\item
\begin{align}
C_{j} (\lambda) := 
\begin{bmatrix} 
C(\lambda) 	& I_2                        			\\
            	& C(\lambda) & \ddots 			\\
            	&                    & \ddots & I_2 		\\ 
            	&                    &            & C(\lambda) 
\end{bmatrix} \in \mat{2j}{\bb{R}}; 												\label{Ck_lambda} 	
\end{align}
\item
\begin{align} 
C(\lambda) := 
\begin{bmatrix} 
\Re{(\lambda)} & \Im{(\lambda)} \\ 
-\Im{(\lambda)} & \Re{(\lambda)}  
\end{bmatrix} \in \mat{2}{\bb{R}}; 												\label{C_lambda}	
\end{align}
\item $\Im{(\lambda_k)} = 0$, $k=1, \dots, r$; and 
\item $\Im{(\lambda_k)} \neq 0$, $k=r+1, \dots, r+c$.
\end{enumerate}
\end{theorem}

\begin{proposition}[{\cite[Corollary 2.11]{mpt2014}}] \label{prop:rjcf_cor} 
Let $\lambda \in \bb{C}$, $\lambda \neq 0$, and let $f$ be a function defined on the spectrum of $\jordan{k}{\lambda} \oplus \jordan{k}{\bar{\lambda}}$.  For $j$ a nonnegative integer, let $f^{(j)}_\lambda$ denote $f^{(j)}(\lambda)$. If $C_k (\lambda)$ and $C(\lambda)$ are defined as in \eqref{Ck_lambda} and \eqref{C_lambda}, respectively, then
\begin{align*} 
f(C_{k} (\lambda)) = 
\begin{bmatrix} 
f(C_\lambda) & f'(C_\lambda) & \dots   & \frac{f^{(k-1)}(C_\lambda)}{(k-1)!}  	\\
		 & f(C_\lambda) & \ddots & \vdots 					\\
		 & 		     & \ddots & f'(C_\lambda)				\\
		 & 		     & 	         & f(C_\lambda)
\end{bmatrix} \in \mat{2k}{\bb{C}},
\end{align*}
and, moreover, 
\begin{align*} 
f ( C_{k} (\lambda) ) = 
\begin{bmatrix} 
C(f_\lambda) & C(f'_\lambda) & \dots & C \left( \frac{f^{(k-1)}_\lambda}{(k-1)!} \right)	\\
& C(f_\lambda) & \ddots & \vdots 									\\
& & \ddots & C(f'_\lambda) 										\\
& & & C(f_\lambda)
\end{bmatrix} \in \mat{2k}{\bb{R}}
\end{align*}
if and only if $\overline{f_\lambda^{(j)}} = f_{\bar{\lambda}}^{(j)}$.
\end{proposition}

We recall the Perron-Frobenius theorem for positive matrices (see \cite[Theorem 8.2.11]{hj1990}).

\begin{theorem} \label{thm:pf} 
If $A \in \mat{n}{\bb{R}}$ is positive, then 
\begin{enumerate}[label=(\roman*)]
\item $\rho := \sr{A} > 0$;
\item $\rho \in \sig{A}$;
\item there exists a positive vector $x$ such that $Ax = \rho x$;
\item $\rho$ is a simple eigenvalue of $A$; and 
\item $|\lambda| < \rho$ for every $\lambda \in \sig{A}$ such that $\lambda \neq \rho$. 
\end{enumerate}
\end{theorem}

One can verify that the matrix 
\begin{align} 
B = \begin{bmatrix} 2 & 1 \\ 2 & -1 \end{bmatrix} \label{matB}
\end{align} 
satisfies properties (i) through (v) of \hyp{thm:pf}{Theorem}, but obviously contains a negative entry. This motivates the following concept.

\begin{definition} 
{\rm A matrix $A \in \mat{n}{\bb{R}}$ is said to \textit{possess the strong Perron-Frobenius property} if $A$ satisfies properties (i) through (v) of \hyp{thm:pf}{Theorem}.}
\end{definition}

It can also be shown that the matrix $B$ given in \eqref{matB} satisfies $B^k > 0$ for $k \geq 4$, which leads to the following generalization of positive matrices. 

\begin{definition}
{\rm A matrix $A \in \mat{n}{\bb{R}}$ is \textit{eventually positive} if there exists a nonnegative integer $p$ such that $A^k > 0$ for all $k \geq p$.} 
\end{definition}

The following theorem relates the strong Perron-Frobenius property with eventually positive matrices (see \cite[Lemma 2.1]{h1981}, \cite[Theorem 1]{jt2004}, or \cite[Theorem 2.2]{n2006}).

\begin{theorem} \label{evpos_thm}
A real matrix $A$ is eventually positive if and only if $A$ and $A^\top$ possess the strong Perron-Frobenius property. 
\end{theorem}

\section{Main Results}

Before we state our main results, we begin with the following definition.

\begin{definition}
{\rm A function $f: \bb{C} \longrightarrow \bb{C}$ defined on a self-conjugate domain $\mathcal{D}$, $\mathcal{D} \cap \bb{R}^+ \neq \emptyset$, is called \textit{Frobenius}\footnote{We use the term `Frobenius' given that such a function preserves \textit{Frobenius multi-sets}, introduced by Friedland in \cite{f1978}.} if 
\begin{enumerate}[label=(\roman*)]
\item \label{item:selfconjugacy} $\overline{f(\lambda)} = f(\bar{\lambda})$, $\lambda \in \mathcal{D}$; 
\item \label{item:mod} $|f(\lambda)| < f(\rho)$, whenever $|\lambda|<\rho$, and $\lambda$, $\rho \in \mathcal{D}$.
\end{enumerate}}
\end{definition}

\begin{rem}
{\rm Condition \ref*{item:selfconjugacy} implies $f(r) \in \bb{R}$, whenever $r \in \mathcal{D} \cap \bb{R}$; and condition \ref*{item:mod} implies $f(r) \in \bb{R}^+$, whenever $r \in \mathcal{D} \cap \bb{R}^+$.}
\end{rem}

The following theorem is our first main result. 

\begin{theorem} \label{mainresultdiag}
Let $A \in \mat{n}{\bb{R}}$ and suppose that $A$ is diagonalizable and possesses the strong Perron-Frobenius property. If $f: \bb{C} \longrightarrow \bb{C}$ is a function defined on the spectrum of $A$, then $f(A)$ possesses the strong Perron-Frobenius property if and only if $f$ is Frobenius.
\end{theorem}

\begin{proof}
Suppose that $f$ is Frobenius. For convenience, denote by $f_\lambda$ the scalar $f(\lambda)$. Following \hyp{thm:rjcf}{Theorem} and \hyp{prop:rjcf_cor}{Proposition}, the matrix
\begin{align}  
f(A)  = 
R \begin{bmatrix} 
f_{\sr{A}} & & 					\\
& \bigoplus_{k=2}^{r} f_{\lambda_k} & 	\\ 
& & \bigoplus_{k = r + 1}^{r + c} C (f_{\lambda_k})  
\end{bmatrix} \inv{R}, 													\label{fofAdiag} 
\end{align}
where $R= 
\begin{bmatrix}
x & R'
\end{bmatrix},~x>0$,
is real. If $\sig{A} = \{ \sr{A}, \lambda_2,\dots, \lambda_n \}$, then $\sig{f(A)} = \{ f_{\sr{A}},f_{\lambda_2},\dots, f_{\lambda_n} \}$ (see, e.g., \cite{h2008}[Theorem 1.13(d)]) and because $f$ is Frobenius, it follows that $|f_{\lambda_k}| < f_{\sr{A}}$ for $k = 2,\dots, n$. Moreover, from \eqref{fofAdiag} it follows that $f(A)x = f_{\sr{A}} x$. Thus, $f(A)$ possesses the strong Perron-Frobenius property.  

Conversely, if $f$ is not Frobenius, then the matrix $f(A)$, given by \eqref{fofAdiag}, is not real (e.g, $\exists \lambda \in \sig{A}$, $\lambda \in \bb{R}$ such that $f(\lambda) \not \in \bb{R}$), or $f(A)$ does not retain the strong Perron-Frobenius property (e.g., $\exists \lambda \in \sig{A}$ such that $|f(\lambda)| \geq f(\sr{A})$). 
\end{proof}

\begin{ex}
{\rm \hyp{table:frobfunctions}{Table} lists examples of Frobenius functions for diagonalizable matrices that possess the strong Perron-Frobenius property. 
\begin{table}[H]
\begin{center}
\begin{tabular}{c|c}
$f$ & $\mathcal{D}$										\\
\hline
$f(z) = z^p$, $p \in \bb{N}$ & $\bb{C}$ 							\\
$f(z) = |z|$ & $\bb{C}$										\\
$f(z) = z^{1/p}$, $p \in \bb{N}$, $p$ even & $\{z \in \bb{C}: z \not \in \bb{R}^- \}$  \\
$f(z) = z^{1/p}$, $p \in \bb{N}$, $p$ odd, $p>1$ & $\bb{C}$  				\\
$f(z) = \sum_{k=0}^n a_k z^k$, $a_k >0$ & $\bb{C}$					\\
$f(z) = \exp{(z)}$ & $\bb{C}$
\end{tabular}
\caption{Examples of Frobenius functions.} \label{table:frobfunctions}
\end{center}
\end{table}}
\end{ex}

For matrices that are not diagonalizable, i.e., possessing Jordan blocks of size two or greater, given \eqref{fjb} it is reasonable to assume that $f$ is complex-differentiable, i.e., \textit{analytic}. We note the following result, which is well known (see, e.g., \cite{bc2013}, \cite{c1995}, \cite[Theorem 3.2]{hmmt2005}, \cite{l1999}, or \cite{r1987}).

\begin{theorem}[Reflection Principle] \label{thm:refprinc}
Let $f$ be analytic in a self-conjugate domain $\mathcal{D}$ and suppose that $I := \mathcal{D} \cap \bb{R} \neq \emptyset$. Then $\overline{f(\lambda)} = f( \bar{\lambda})$ for every $\lambda \in \mathcal{D}$ if and only if $f(r) \in \bb{R}$ for all $r \in I$.
\end{theorem}

The Reflection Principle leads immediately to the following result.

\begin{corollary}
An analytic function $f: \bb{C} \longrightarrow \bb{C}$ defined on a self-conjugate domain $\mathcal{D}$, $\mathcal{D} \cap \bb{R}^+ \neq \emptyset$, is \textit{Frobenius} if and only if
\begin{enumerate}[label=(\roman*)]   
\item \label{item:real2} $f(r) \in \bb{R}$, whenever $r \in \mathcal{D} \cap \bb{R}$; and
\item \label{item:modulus2} $|f(\lambda)| < f(\rho)$, whenever $|\lambda|<\rho$ and $\lambda$, $\rho \in \mathcal{D}$.
\end{enumerate}
\end{corollary}

\begin{lemma} \label{lem:analytic}
Let $f$ be analytic in a domain $\mathcal{D}$ and suppose that $I := \mathcal{D} \cap \bb{R} \neq \emptyset$. If $f(r) \in \bb{R}$ for all $r \in I$, then $f^{(j)}(r) \in \bb{R}$ for all $r \in I$ and $j \in \bb{N}$..
\end{lemma}

\begin{proof}
Proceed by induction on $j$: when $j=1$, note that, since $f$ is analytic on $\mathcal{D}$, it is \textit{holomorphic} (i.e., complex-differentiable) on $\mathcal{D}$. Thus,   
\begin{align*}
f'(r) := \lim_{z \rightarrow r} \frac{f(z) - f(r)}{z-r}
\end{align*}
exists for all $z \in \mathcal{D}$; in particular, 
\begin{align*}
f'(r) = \lim_{x \rightarrow r} \frac{f(x) - f(r)}{x - r},~x \in I, 
\end{align*}
and the conclusion that $f'(r) \in \bb{R}$ follows by the hypothesis that $f(x) \in \bb{R}$ for all $x \in I$.

Next, assume that the result holds when $j=k-1>1$. As above, note that $f^{(k)}(r)$ exists and 
 \begin{align*}
f^{(k)}(r) = \lim_{x \rightarrow r} \frac{f^{(k-1)}(x) - f^{(k-1)}(r)}{x - r},~x \in I 
\end{align*}
so that $f^{(k)}(r) \in \bb{R}$.
\end{proof}

\begin{theorem} \label{thm:mainresult}
Let $A \in \mat{n}{\bb{R}}$ and suppose that $A$ possesses the strong Perron-Frobenius property. If $f: \bb{C} \longrightarrow \bb{C}$ is an analytic function defined in a self-conjugate domain $\mathcal{D}$ containing $\sig{A}$, then $f(A)$ possesses the strong Perron-Frobenius property if and only if $f$ is Frobenius.
\end{theorem}

\begin{proof}
Suppose that $f$ is Frobenius. Following \hyp{thm:rjcf}{Theorem}, there exists an invertible matrix $R$ such that 
\begin{align*}
\inv{R} A R =
\begin{bmatrix}
\sr{A} & 															\\ 
	& \bigoplus_{k=2}^{r} J_{n_k} (\lambda_k) 	& 								\\ 
	&							& \bigoplus_{k = r + 1}^{r + c} C_{n_k} (\lambda_k)  
\end{bmatrix},															
\end{align*}
where 
\begin{align*}
R= 
\begin{bmatrix}
x & R'
\end{bmatrix},~x>0.
\end{align*}
Because $f$ is Frobenius, following \hyp{thm:refprinc}{Theorem}, $\overline{f(\lambda)} = f(\bar{\lambda})$ for all $\lambda \in \mathcal{D}$. Since $f$ is analytic, $f^{(j)}$ is analytic for all $j \in \bb{N}$ and, following \hyp{lem:analytic}{Lemma} $f^{(j)}(r) \in \bb{R}$ for all $r \in I$. Another application of \hyp{thm:refprinc}{Theorem} yields that $\overline{f^{(j)}(\lambda)} = f^{(j)}(\bar{\lambda})$ for all $\lambda \in \mathcal{D}$. Hence, following \hyp{prop:rjcf_cor}{Proposition}, the matrix 
\begin{align*}
f(A) =
R
\begin{bmatrix}
f(\sr{A}) 	& 																	\\ 
		& \bigoplus_{k=2}^{r} f(J_{n_k} (\lambda_k)) 	& 									\\ 
		&								& \bigoplus_{k = r + 1}^{r + c} f(C_{n_k}(\lambda_k))  
\end{bmatrix} \inv{R}															
\end{align*}
is real and possesses the strong Perron-Frobenius property.

The proof of the converse is identical to the proof of the converse of \hyp{mainresultdiag}{Theorem}. 
\end{proof}

\begin{corollary}
Let $A \in \mat{n}{\bb{R}}$ and suppose that $A$ is eventually positive. If $f: \bb{C} \longrightarrow \bb{C}$ is an analytic function defined in a self-conjugate domain $\mathcal{D}$ containing $\sig{A}$, then $f(A)$ is eventually positive if and only if $f$ is Frobenius.
\end{corollary}

\begin{proof}
Follows from \hyp{thm:mainresult}{Theorem} and the fact that $f(A^\top) = (f(A))^\top$ (\cite[Theorem 1.13(b)]{h2008}).
\end{proof}

\begin{rem}
{\rm Aside from the function $f(z) = |z|$, which is nowhere differentiable, every function listed in \hyp{table:frobfunctions}{Table} is analytic and Frobenius.}
\end{rem}

\section{Acknowledgements}

I gratefully acknowledge Hyunchul Park for discussions arising from this research and the anonymous referees for their helpful suggestions.

\bibliographystyle{abbrv}
\bibliography{laabib,crs}

\end{document}